\documentclass[11 pt]{article}
\usepackage{amsmath, amssymb}
\usepackage{amsthm}

\theoremstyle{plain} 
\newtheorem{theorem}{\sc \bf Theorem}[section]
\newtheorem{lemma}[theorem]{\noindent\sc \bf Lemma}
\newtheorem{corollary}[theorem]{\noindent\sc \bf Corollary}

\theoremstyle{plain}

\newtheorem{example}[theorem]{\sc \bf Example}

\def \RR {\mathbb{R}}

\makeatother

\title{\uppercase {Classification of Ordered Type Soliton Metric Lie Algebras by a Computational Approach} } 

\author{H\"{u}lya Kad{\i}o\u{g}lu\\ \small Department of Mathematics Education , Y{\i}ld{\i}z Technical University, Turkey\\ \small hkadio@yildiz.edu.tr}
\date{}
\begin{document}
   \maketitle
   
   \begin{abstract} 
   In this study, we classify some soliton nilpotent Lie algebras and possible candidates in dimension 8 and 9 up to isomorphy. We focus on $1<2<...<n$ type of derivations where $n$ is the dimension of the Lie algebras.
 We present algorithms to generate possible algebra structures.  
   \end{abstract}

\section{Introduction} 

In this paper, we compute and clasify $n$ dimensional $(n=8,9)$ nilsoliton metric Lie algebras with eigenvalue type
$1<2<....<n$, which will be called "ordered type of Lie algebra" throughout this paper. We use a computational procedure implemented in a computer
algebra system MATLAB to achieve this goal. In the literature, {\em six} dimensional nilpotent Lie algebras have been 
classified by algorithmic approaches (\cite{Graafb}). In dimension {\em seven} and lower, 
nilsoliton metric Lie algebras have been classified (\cite{Lauretb}, \cite{Will}, \cite{Kad}, \cite{Culmaa}, 
\cite{Culmab}, \cite{Lauretd}, \cite{Nikoa}, \cite{Arroyo}). Summary and details of some other classifications can be 
found in \cite{Lauretc}.
In our paper,  we focus on dimensions eight and nine. We note that we have found that our algorithm gives consistent 
results with the literature in lower dimensions. We use a computational procedure that is similar to the one that we 
have used in our previous paper \cite{Kad}. 

In our previous paper, we classified all the soliton and nonsoliton metric Lie algebras where the corresponding Gram 
matrix is invertible and of dimension $7$ and $8$ up to isomorphism. 
If corresponding Gram matrix is invertible, then the soliton metric condition $Uv=[1]$ has a unique solution. 
So in this case,  it is easy to check if the algebra is soliton or not. But in noninvertible case, there is more than one solution.
 Therefore it is hard to guess if one of the solutions provides the soliton condition without solving Jacobi identity which is nonlinear. 
On the other hand, it may be easy if we can eliminate some algebras which admits a derivation $D$ that does not
 have ordered eigenvalues without solving the following  soliton metric condition $Uv=[1]$. 
For this,  we prove that if the nilpotent Lie algebra admits a soliton metric with corresponding Gram matrix 
of $\eta$ being noninvertible, all the solutions of $Uv=[1]$  has a unique derivation. This theorem allows us to omit 
several cases that come from non ordered eigenvalues without considering Jacobi identity. 

This paper is organized as follows. In section \ref{pre}, we provide some preliminaries that we use for our 
classifications. 
In section \ref{theory}, we give specific Jacobi identity conditions for Lie algebras up to dimension {\em nine}.
 This allows us to decide whether the Lie algebra has a soliton metric or not. 
In  Section \ref{alg}, we give details of our classifications with specific examples  and provide algorithmic 
procedures.
Section \ref{conc} contains  our concluding remarks. 

\section{Preliminaries} \label{pre}

Let $(\eta_{\mu},Q)$ be a metric algebra , where $\mu \in \Lambda^2 \eta \otimes \eta^*$. Let $B=\{X_i\}_{i=1}^n$ be a $Q$-orthonormal basis of $\eta_{\mu}$ (We always assume that basis are ordered). The nil-Ricci endomorphism $Ric_{\mu}$ is defined as $<Ric_{\mu}X,Y>=ric_{\mu}(X,Y)$, where

\begin{equation}
ric_{\mu}(X,Y)=-\frac{1}{2} \displaystyle\sum \limits_{i=1}^n <[X,X_i],[Y,X_i]>+\frac{1}{4} \displaystyle\sum \limits_{i=1}^n <[X_i,X_j],X><[X_i,X_j],Y>
\end{equation}
for $X,Y\in \eta$ (We often write an inner product $Q(.,.)$ as $<.,.>$.). When $\eta$ is a nilpotent Lie algebra, the nil-Ricci endomorphism is the Ricci endomorphism. If all elements of the basis are eigenvectors for the nil-Ricci endomorphism $Ric_{\mu}$, we call the orthonormal basis a Ricci eigenvector basis.
\newline

Now we define some combinatorial objects associated to a set of integer triples $\Lambda \subset \{(i,j,k)|1 \leq i,j,k\leq n\}$. For $ 1 \leq i,j,k\leq n$, define $1 \times n$ row vector $y_{ij}^k$ to be $\epsilon_i^T+\epsilon_j^T-\epsilon_k^t$, where $\{\epsilon\}_{i=1}^n$ is the standart orthonormal basis for $\RR ^n$. We call the vectors in $\{y_{ij}^k|(i,j,k)\in \Lambda\}$ {\it root vectors} for $\Lambda$. Let $y_1,y_2,...,y_m$ (where $m=\left|\Lambda\right|$) be an enumeration of the root vectors in dictionary order. We define {\it root matrix} $Y_{\Lambda}$ for $\Lambda$ to be the $m \times n$ matrix whose rows are the root vectors $y_1,y_2,...,y_m$. The {\it Gram matrix} $U_{\Lambda}$ for $\Lambda$ is the $m \times m$ matrix defined by $U_{\Lambda}=Y_{\Lambda}Y_{\Lambda}^T$; the $(i,j)$ entry of $U_{\Lambda}$  is the inner product of the {\it i} th and {\it j} th root vectors. It is easy to see that $U$ is a symmetric matrix. It has the same rank as the root matrix, i.e. $Rank(U_{\Lambda})=Rank(Y_{\Lambda})$. Diagonal elements of $U$ are all three, and the off-diagonal entries of $U$ are in the set $\{-2,-1,0,1,2\}$. For more information, see \cite{Paynea}. Let $D$ has distinct real positive eigenvalues, and let $\Lambda$ index the structure constants for $\eta$ with respect to eigenvector basis $B$.   If  $(i_1,j_1,k_1) \in \Lambda$ and $(i_2,j_2,k_2) \in \Lambda$, then $<y_{i_1,j_1}^{k_1}, y_{i_2,j_2}^{k_2}> \neq 2$. Thus $U$ does not contain two as an entry \cite{Kad}.
\begin{lemma} \label{lemma1}
Let $(\eta,Q)$ be an $n$-dimensional inner product space, and let $\mu$ be an element of $\Lambda^2 \eta ^*\otimes \eta$. Soppose that $\eta_{\mu}$ admits a symmetric derivation $D$ having $n$ distinct eigenvalues $0<\lambda_1<\lambda_2<...<\lambda_n$ with corresponding orthonormal eigenvectors $X_1,X_2,...,X_n$. Let $\alpha_{ij}^k$ denote the structure constants for $\eta$ with respect to the ordered basis $B=\{X_i\}_{i=1}^n$. Let $1\leq i<j\leq n$. Then

\begin{enumerate}
\item If there is some $k\in \{1,2,...,n\}$ such that $\lambda_k=\lambda_i+\lambda_j$, then $[X_i,X_j]$ is a scalar multiple of $X_k$; otherwise $X_i$ and $X_j$ commutes.
\item $\alpha_{ij}^k \neq 0$ if and only if $X_k \in [\eta_{\mu},\eta_{\mu}]$.
\end{enumerate}

\end{lemma}

\begin{theorem} \label{theorem2} (\cite{Paynea}) Let $\eta$ be a vector space, and let $ B={X_i}_{i=1}^n$ be a basis for $\eta$. Suppose that a set of nonzero structure constants $\alpha_{i,j}^k$ relative to $B$, indexed by $\Lambda$, defines a skew symmetric product on $\eta$. Assume that if $(i,j,k)\in \Lambda$, then $i<j<k$. Then the algebra is a Lie algebra if and only if whenever there exists $m$ so that the inner product of root vectors $<y_{ij}^l,y_{lk}^m>=-1$ for triples $(i,j,l)$ and $(l,k,m)$ or $(k,l,m)$ in $\Lambda$, the equation

\begin{equation}
\displaystyle\sum_{s<m} \alpha_{i,j}^s \alpha_{s,k}^m +\alpha_{j,k}^s \alpha_{s,i}^m+ \alpha_{k,i}^s \alpha_{s,j}^m=0 \label{JI}
\end{equation}
holds. Furthermore, a term of form $\alpha_{i,j}^l \alpha_{l,k}^m$ is nonzero if and only if $<y_{i,j}^l,y_{l,k}^m>=-1$
\end{theorem} 

\begin{theorem} \label{theorem3} \cite{Paynea}
Let $(\eta_{\mu},Q)$ be a metric algebra and $B=\{X_i\}_{i=1}^n$ be a Ricci eigenvector basis for $\eta_{\mu}$. Let $Y$ be the root matrix for $\eta_{\mu}$. Then the eigenvalues of the nil-Ricci endomorphism are given by 
\begin{equation}
Ric_{\mu}^{B}=-\frac{1}{2}Y^Tv   \label{ricci}
\end{equation}
where $v=[\alpha^2]$.
\end{theorem}

\begin{theorem} \label{theorem1}(\cite{Paynea} and \cite{Kad}) 
Let $(\eta,Q)$ be a nonabelian metric algebra with Ricci eigenvector basis $B$. The following are equivalent:
\begin{enumerate}
\item $(\eta_{\mu},Q)$ satisfies the nilsoliton condition with nilsoliton constant $\beta$.
\item The eigenvalue vector $V_D$ for $D=Ric-\beta Id$ with respect to $B$ lies in the kernel of the root matrix for $(\eta_{\mu},Q)$ with respect to $B$.
\item For noncommuting eigenvectors $X$ and $Y$ for the nil-Ricci endomorphism with eigenvalues $\kappa_X$ and $\kappa_Y$, the bracket $[X,Y]$ is an eigenvector for the nil-Ricci endomorphism with eigenvalue $\kappa_X+\kappa_Y-\beta$.
\item $\beta=y_{ij}^k Ric$ for all $(i,j,k)$ in $\Lambda(\eta_{\mu},B)$.
\end{enumerate}
\end{theorem}

\begin{theorem} \label{invertible} \cite{Kad}
Let $\eta$ be an $n-$ dimensional nonabelian nilpotent Lie algebra which admits a derivation $D$ having distinct real positive eigenvalues. Let $B$ a basis consisting of eigenvectors for the derivation $D$, and let $\Lambda$ index the nonzero structure constants with respect to $B$. Let $U$ be the $m \times m$ Gram matrix . If $U$ is invertible, then the following hold:
\begin{itemize}
\item $ \left|\Lambda\right| \leq n-1$
\item If $(i_1,j_1,k_1) \in \Lambda$ and $(i_2,j_2,k_2) \in \Lambda$, then $<y_{i_1,j_1}^{k_2},y_{i_2,j_2}^{k_2}> \neq -1$
\end{itemize}
\end{theorem}

\section{Theory} \label{theory}

This section provides  some theorems and their proofs that allow us to consider fewer cases for our algoritm. 
The following theorem gives a pruning method while Gram matrix is noninvertible. 

\begin{theorem} \label{unique}

Let $\eta$ be an $n$ dimensional nilsoliton metric Lie algebra, and $U$ be the corresponding Gram matrix which is noninvertible.Then $Ker(Y^T)=Ker(U)$. Furhermore
all of the solutions of $Uv=[1]$ correspond to a unique derivation.
\end{theorem}

\begin{proof}

Since rank of a matrix is equal to the rank of its Gram matrix, then $p=Rank(Y)=Rank(U)$. Let $U:\RR ^p \to \RR ^p$ and $Y^T:\RR^p \to \RR ^n$ denote the linear functions (with respect to the standart basis) that correspond to the Gram matrix $U$ and the transpose of the root matrix $Y$ respectively. Since $Rank(U)=Rank(Y)=Rank(Y^T)$ and by rank-nullity theorem, we have 
\begin{equation}
Ker(U)=Ker(Y^T).\label{ker}
\end{equation}
Let $v$ be a particular solution and $v_0$  be the last column of reduced row echelon matrix $[U, [1]]$. Then $v_0$ is also a solution of $Uv_0=[1]$. Therefore $U(v-v_0)=0$, i.e. $(v-v_0) \in Ker(U)$. Using equation (\ref{ker}), then $(v-v_0) \in Ker(Y^T)$. For the solution $v_0$, suppose that we denote $D_0$ for the Nikolayevski derivation, $Ric_0$ for the Ricci tensor, and $\beta_0$ for the soliton constant. Then using the equation (\ref{ricci}), we have 

\begin{equation}
(Ric-Ric_0)= \frac{1}{2}Y^t(v-v_0)=0.
\end{equation}
Then $Ric=Ric_0$. Using Theorem (\ref{theorem1}), we have $\beta=\beta_0$, which implies $D=D_0$.
\end{proof}
\begin{lemma} \label{lemma}
If  nilsoliton metric Lie algebra $\eta$ has ordered type of derivations $1<2<.....<n $, then it's index set $\Lambda$ consists of triples $(i,j,i+j)$.
\end{lemma}
\begin{proof}
If $V_D=(\lambda_1,\lambda_2,...,\lambda_n)^T$ is the eigenvalue vector of $D$ with eigenvector basis $B=\{X_i\}_{i=1}^n$  for $\eta$, then by Theorem \ref{theorem1}, $V_D$ lies in the kernel of $Y$. Thus for each element $(i,j,k) \in \Lambda$, $\lambda_i+\lambda_j-\lambda_k=0$, i.e., $\lambda_i+\lambda_j=\lambda_k$. By Lemma \ref{lemma1}, $[X_i,X_j]=\lambda X_k$ for some $\lambda \in \RR$. Since $\lambda_i=i$ for all $i \in \{1,2,...,n\}$, $k=i+j$ and  $[X_i,X_j]=\lambda X_{i+j}$. 
Hence the index set for ordered type of derivations are of form $(i,j,i+j)$.
\end{proof} 
Next corollary describes the index triples $(i,j,k)$ and the Jacobi identity for algebras with ordered type of derivations.

\begin{corollary} \label{ijk}

The algebra $\eta $ is a Lie algebra if and only if for all pairs of form $(i,j,l)$ and $(l,k,m)$ or $(i,j,l)$ and $(k,l,m)$ in $\Lambda_B$ with $k\notin \{i,j\}$, and for all $m\geq max\{i+3,j+2,5\}$, the following equation holds. 

\begin{equation}
\displaystyle\sum_{3 \leq s <m,s\notin\{i,j,k\}} \alpha_{i,j}^s \alpha_{s,k}^m +\alpha_{j,k}^s \alpha_{s,i}^m+ \alpha_{k,i}^s \alpha_{s,j}^m=0
\end{equation}

If in addition $\lambda_i =i$ for $i=1,.....,n$, then the algebra $\eta$ is a Lie algebra iff for all pairs of form $(i,j,i+j)$ and $(i+j,k,i+j+k)$, or $(i,j,i+j)$ and $(k,i+j,i+j+k)$ in $\Lambda_B$ with $k\notin \{i,j,i+j\}$, and for all $m=i+j+k\geq max\{2i+2,j+2,6\}$, the equation

\begin{equation}
\displaystyle\sum_{4 \leq s <m,s\notin\{i,j\}} \alpha_{i,j}^s \alpha_{s,k}^m +\alpha_{j,k}^s \alpha_{s,i}^m+ \alpha_{k,i}^s \alpha_{s,j}^m=0
\end{equation}
holds.
\end{corollary}

\begin{proof}
\noindent By Theorem 7 of \cite{Paynea}, the algebra $\eta_{\mu}$ defined by $\mu$ is a Lie algebra if and only if whenever there exists $m$ so that $<y_{ij}^l,y_{lk}^m>=-1$ for triples $(i,j,l)$ and $(l,k,m)$ or $(k,l,m)$ in $\Lambda_B$, the Equation \ref{JI}

\begin{equation}
\displaystyle\sum_{s<m} \alpha_{i,j}^s \alpha_{s,k}^m +\alpha_{j,k}^s \alpha_{s,i}^m+ \alpha_{k,i}^s \alpha_{s,j}^m=0
\end{equation}
holds. Furthermore , if $i$,$j$, and $k$ are distinct, the product $\alpha_{ij}^l\alpha_{lk}^m$ is nonzero  if and only if $<y_{ij}^l,y_{lk}^m>=-1$.
\newline

\noindent Suppose that $<y_{ij}^l,y_{lk}^m>=-1$ for $(i,j,l)\in\Lambda_B$, and $(l,k,m)$ or $(k,l,m)$ in $\Lambda_B$. By definition of $\Lambda_B$, we have $i<j$. By Lemma \ref{lemma1}, $j<l$, $l<m$ and $k<m$. Since $i<j<l<m$, we know that $m\geq i+3$. Similarly, $j<l<m$ implies that $j+2 \leq m$. If $i=k$ or $j=k$, then $<y_{ij}^l,y_{lk}^m>=0$, so $i,j$ and $k$ must be distinct. Since $i,j,k$ and $l$ are all distinct and less than $m$, we know $m \geq 5$. Thus an expression of form 
\begin{equation}
\alpha_{i,j}^s \alpha_{s,k}^m +\alpha_{j,k}^s \alpha_{s,i}^m+ \alpha_{k,i}^s \alpha_{s,j}^m 
\end{equation}
is nonzero only if $m\geq max\{i+3,j+2,5\}$, and $k \notin \{i,j\}$.

Suppose that $\lambda_i=i$, and $(i,j,l)$ and $(l,k,m)$ are in the index set. Then from Lemma \ref{lemma}, $l=i+j$ which implies $m=i+j+k$. We know that $1 \leq i<j<l<m$. Then, since $j \geq i+1$ and $k \geq 1$, we have $2i+2 \leq i+j+k=m$. Since $2i+2 \leq m$, $m=5$ implies $i=1$. So there is no possible $(i,j,k)$, where all $i,j,k$ are distinct and $i<j<m$ with $i+j+k$. Thus if $m=5$, then $<y_{ij}^l,y_{lk}^m> \neq -1$. Therefore, if  $\lambda_i=i$, an expression of form 
\begin{equation}
\alpha_{i,j}^s \alpha_{s,k}^m +\alpha_{j,k}^s \alpha_{s,i}^m+ \alpha_{k,i}^s \alpha_{s,j}^m  \label{expression}
\end{equation}
is nonzero only if $m\geq max\{2i+2,j+2,6\}$, and $k \notin \{i,j\}$.

By Lemma \ref{lemma1}, $X_1$ and $X_2$ are in $[\eta_{\mu},\eta_{\mu}]^{\bot}$, so $\alpha_{rt}^s\neq 0$ which implies $s \geq 3$ and $r\neq s$, $t \neq s$. Therefore all expressions in Equation \ref{expression} with $s<3$ or $s \in \{i,j,k\}$ are identically zero and may be omitted from the summation for any $m$.
\end{proof}
Next corollary describes some equations in the structure constants of a nilpotent metric Lie algebra that are equivalent to the Jacobi identity. All the terms $\alpha_{ij}^s\alpha_{sk}^m$ in the following equations leads $-1$ entry in $U$. Therefore, following equations are useful for noninvertible case since there is no $-1$ entry in the Gram matrix for the invertible case.

\begin{corollary} \label{cor1}
Let $(\eta,<,>)$ be an $n$-dimensional inner product space where $n\leq 9$, and $\mu$ be an element of $\Lambda^2n*\otimes n$. Suppose that the algebra $\eta_{\mu}$ defined by $\mu$ admits a symmetric derivation $D$ having $n$ eigenvalues $1<2<...<n$ with corresponding orthonormal eigenvectors $X_1, X_2,...X_n$. $\alpha_{i,j}^k$ denote the structure constants for $\eta$ with respect to the ordered basis $B=\{X_i\}_{i=1}^n$, and let $\lambda_B$ index the nonzero structure constants as defined in Equation (\ref{JI}).
The algebra $\eta_{\mu}$ is a Lie algebra if and only if 

\begin{eqnarray}
&\alpha_{13}^4 \alpha_{42}^6+\alpha_{23}^5 \alpha_{51}^6=0,\hspace{2mm}, and \label{1st}\\ 
&\alpha_{12}^3 \alpha_{3,4}^7 -\alpha_{14}^5 \alpha_{52}^7-\alpha_{24}^6 \alpha_{61}^7=0\hspace{2mm}, and \label{2nd}\\ 
&\alpha_{12}^3 \alpha_{35}^8+\alpha_{14}^5 \alpha_{53}^8+\alpha_{15}^6 \alpha_{62}^8+ \alpha_{25}^7\alpha_{71}^8 + \alpha_{34}^7\alpha_{71}^8 =0 \hspace{2mm} ,and \label{3rd} \\ 
&\alpha_{12}^3 \alpha_{36}^9 +\alpha_{13}^4 \alpha_{45}^9 -\alpha_{23}^5 \alpha_{45}^9 -\alpha_{24}^6 \alpha_{36}^9 -\alpha_{34}^7 \alpha_{27}^9 -\alpha_{35}^8 \alpha_{18}^9 =0. \label{4th}
\end{eqnarray}

holds.
\end{corollary}

\begin{proof}
Following the Lemma \ref{lemma}, the index set consists of elements of form $(i,j,i+j)$. Therefore the number $s$ equals to $i+j$ in the expression $\alpha_{i,j}^s \alpha_{s,k}^m$. From the previous corollary,
\begin{equation}
\alpha_{i,j}^s \alpha_{s,k}^m +\alpha_{j,k}^s \alpha_{s,i}^m+ \alpha_{k,i}^s \alpha_{s,j}^m  
\end{equation}
is nonzero only if $m\geq 6$, and $k \notin \{i,j\}$. Also $k \neq i+j$, otherwise  $\alpha_{s,k}^m = 0$.

If $m=6$, $2i+2 \leq m$ implies that $i \leq 2$, i.e. possible numbers for "$i$" are $1$ and $2$.
\
\begin{table}[ht]
\centering  
\begin{tabular}{|c|c|c|c|c|} 
 \hline                       
Case & i & j & k &PT\\ [0.5ex] 
\hline                  
a) & 1 & 2 & 3 & - \\ 
\hline
b) & 1 & 3 & 2 & \checkmark \\
\hline
c) & 1 & 4 & 1 & - \\
\hline
d) & 2 & 3 & 1 & \checkmark \\
\hline
e) & 2 & 4 & 0 & - \\ [1ex]      
\hline 
\end{tabular}
\caption{Possible (i,j,k) for m=6} 
\label{table1} 
\end{table}
Possible and not possible $(i,j,k)$ triples, which are being used in 
\begin{equation}
\displaystyle\sum_{3 \leq s <m,s\notin\{i,j,k\}} \alpha_{i,j}^s \alpha_{s,k}^m +\alpha_{j,k}^s \alpha_{s,i}^m+ \alpha_{k,i}^s \alpha_{s,j}^m=0
\end{equation}
and where $i+j+k=6$ are illustrated in Table \ref{table1}. The notation in the table is as follows: $\checkmark :=Yes$, $-:= No$, $PT:=$ Possible Triple.
\newline  
 
\noindent In the case a), $i+j=k$, then it is  not a possible triple. 
In the case b), $i<j$, $i,j,k$ are distinct, and $k \neq i+j$. So it is a possible triple.
In the case c), $i=k$, so it is not a possible triple. 
In the case d), $i<j$, all $i,j,k$ are distinct $i+j \neq k$. Thus it is a possible triple.
In the case e), $k$ is not a natural number, so it is not a possible triple.
Therefore only possible  $(i,j,k)$ triples are $(1,3,2)$ and $(2,3,1)$.
Triples $(1,3,2)$  and $(2,3,1)$ corresponds to nonzero products $\alpha_{13}^4\alpha_{42}^6$ and $\alpha_{23}^5\alpha_{5,1}^6$ respectively.  Using the skew-symmetry, the Equation (\ref{JI}) turns into following equation:
\begin{equation}
\alpha_{13}^4 \alpha_{24}^6+\alpha_{23}^5 \alpha_{15}^6=0,
\end{equation}
which gives Equation (\ref{1st}).
Using the same procedure for $m=7$, possible $(i,j,k)$ triples are $(1,2,4),(1,4,2)$ and $(2,4,1)$ which leads $\alpha_{12}^3\alpha_{34}^7$, $\alpha_{14}^5\alpha_{52}^7$ and $\alpha_{24}^6\alpha_{61}^7$ respectively. Therefore the Equation (\ref{2nd}) is obtained. Equations (\ref{3rd}) and (\ref{4th}) can be obtained by the same way. 
\end{proof}

As an illustration, we show how to use the results of this section in the following example. 
\begin{example} 
Let $\eta$ be an $8$-dimensional algebra with nonzero structure constants relative to eigenvector basis $B$ indexed by 
\begin{equation}
\Lambda=\{(1,2,3), (1,3,4), (1,4,5), (1,6,7), (2,3,5), (2,6,8), (3,4,7), (3,5,8).\} 
\end{equation}

Computation shows that the structure vector $[\alpha ^2]$ is a solution  to $Uv=[1]_{8 \times 1}$ if and only if it is of form 
\begin{equation*}
[\alpha ^2]=\
\begin{pmatrix}
(\alpha_{12}^3)^2\\
(\alpha_{13}^4)^2\\
(\alpha_{14}^5)^2\\
(\alpha_{16}^7)^2\\
(\alpha_{23}^5)^2\\
(\alpha_{26}^8)^2\\
(\alpha_{34}^7)^2\\
(\alpha_{35}^8)^2
\end{pmatrix} =\
\begin{pmatrix}
-3/17\\
10/17\\
15/17\\
-8/17\\
-5/17\\
12/17\\
7/17\\
0\\
0
\end{pmatrix}+X.\
\begin{pmatrix}
1\\
0\\
-1\\ 
0\\
0\\
-1\\
0\\
1\\ 
0
\end{pmatrix}+Y.
\begin{pmatrix}
1\\
-1\\
-1\\ 
1\\
1\\
-1\\
-1\\
0\\ 
1
\end{pmatrix}
\end{equation*}

Equation (\ref{3rd}) from the previous corollary leads 
\begin{eqnarray}
\begin{gathered}
(-\frac{3}{17}+X+Y)\cdot Y=(\frac{15}{17}-X-Y)\cdot Y\\ 
\Rightarrow X+Y=\frac{9}{17}.
\end{gathered}
\end{eqnarray} 
Moreover, using Equation (\ref{2nd}), we find $X=\frac{3}{17}$ and $Y=\frac{6}{17}$, which means $(\alpha_{16}^7)^2=-\frac{2}{17}$. Thus $\eta$ is not a Lie algebra. 
\end{example}


\section{Algorithm and Classifications} \label{alg}
In this section, we  describe our computational procedure and give the results in dimension 8 and 9.
\newline

\subsection{Algorithm}
Now we describe the algorithm.
The following algorithm can be used for both invertible and noninvertible case. 

{\bf Input.} The input is the integer $n$ which represents the dimension.
\newline

{\bf Output.} Outputs Wsoliton and Uninv listing characteristic vectors for index sets of $\Lambda$ of $\Theta_n$. The matrix Wsoliton has as its rows all possible characteristic vectors for canonical index sets $\Lambda$ for nilpotent Lie algebras of dimension $n$ with ordered type nonsingular nilsoliton derivation whose canonical Gram matrix $U$ is invertible. The matrix Uninv has as its rows all possible characteristic vectors for canonical index sets $\Lambda$ for nilpotent Lie algebras of dimension $n$ with ordered type nonsingular nilsoliton derivation whose canonical Gram matrix $U$ is noninvertible. In the dimensions $8$ and $9$ there is no example for invertible case. Thus Wsoliton is an empty matrix. Therefore we give the algorithm for the noninvertible case.
\newline

\noindent \underline{\bf Algorithm for the noninvertible Case} \\
\begin{itemize}
\item  Enter the dimension $n$.
\item Compute the matrix $Z_n$.
\item Compute the matrix $W$.
\item Delete all rows of $W$ containing abelian factor which is the row that represents direct sums of Lie algebras. 
\item Remove all rows of $W$ such that the canonical Gram matrix $U$ associated to the index set $\Lambda$ is invertible.
\item Define eigenvalue vector $v_D=\begin{pmatrix} 1 & 2 & 3 & ....& n \end{pmatrix} ^T$ in dimension $n$.
\item Remove all rows of $W$ if $v(i)=v_0(i) \leq 0$ where $v$ is the general solution of $U_{\Lambda} v=[1]_{m \times 1}$ and $v_0$ is the vector that we have defined in the proof of Theorem \ref{unique}.
\item Remove all the rows of $W$ such that the corresponding algebra does not have a derivation of eigenvalue type $1<2<...<n$. 
\item Remove all the rows of $W$ such that the corresponding algebra does not satisfy Jacobi identity condition, which is obtained in Corollary \ref{cor1}.
\end{itemize}

After this process, we solve nonlinear systems which follows from Jacoby identity. In order to see how the algorithm works, we give the following example for n=6.

\begin{example}
Let $n=6$. Then 
\begin{equation*}
\Theta_6=\{(1,2,3),(1,3,4),(1,4,5),(1,5,6),(2,3,5),(2,4,6)\}
\end{equation*}
So, matrix $Z_6$ is $6 \times 3$ of form
\begin{equation*}
Z_6=\begin{pmatrix}
1&2&3\\
1&3&4\\
1&4&5\\
1&5&6\\
2&3&5\\
2&4&6
\end{pmatrix}
\end{equation*}
Since $ |\Theta_6|=6$, the matrix $W$ is of size $2^6 \times 6$:
\begin{equation*}
W_{\Theta}=\begin{pmatrix}
0&0&0&0&0&0\\
0&0&0&0&0&1\\
0&0&0&0&1&0\\
.&.&.&.&.&.\\
.&.&.&.&.&.\\
.&.&.&.&.&.\\
1&1&1&1&1&1
\end{pmatrix}
\end{equation*}
The first row of $W_{\theta}$ represents empty matrix, row two represents the subset $\{(2,4,6)\}$ of $\Theta_6$, etc. 
Eliminating rows that represents direct sums, we have 33 rows in $W$ matrix. Therefore these new $W$ does not include Lie algebras that can be written as direct sums. These algebras corresponds both invertible and noninvertible gram matrices. There is no example for the invertible case. For the noninvertible case, there is one ordered type nilsoliton metric Lie algebra $\eta$. If $B$ is the eigenvector basis for $\eta$, whose nonzero structure constants are indexed by
\begin{equation*}
\Lambda_{\eta}=\{(1,2,3), (1,3,4), (1,4,5), (1,5,6), (2,3,5), (2,4,6)\}
\end{equation*}
where $B$ is the eigenvector basis for $\eta$. Computation shows that the structure vector $[\alpha ^2]$ is a solution  to $Uv=[1]_{6 \times 1}$ if and only if it is of form 
\begin{equation*}
[\alpha ^2]=\
\begin{pmatrix}
(\alpha_{12}^3)^2\\
(\alpha_{13}^4)^2\\
(\alpha_{14}^5)^2\\
(\alpha_{15}^6)^2\\
(\alpha_{23}^5)^2\\
(\alpha_{24}^6)^2
\end{pmatrix} =\frac{1}{143}\
\begin{pmatrix}
2\\
1\\
2\\
5\\
5\\
0
\end{pmatrix}+\frac{t}{143}\
\begin{pmatrix}
0\\
1\\
0\\ 
-1\\
-1\\
1
\end{pmatrix}
\end{equation*}.
By Corollary \ref{cor1}, $\eta$ satisfies Equation \ref{1st}. Solving the equation for $t$, we find that 
\begin{eqnarray}
(\alpha_{13}^4\alpha_{24}^6)^2&=&(\alpha_{23}^5\alpha_{15}^6)^2 \nonumber\\
(11+t)t&=&(55-t)^2 \nonumber\\
t&=&25. \nonumber
\end{eqnarray} 
After rescaling and solving for structure constants from $[\alpha ^2]$, we see that letting 
\begin{equation}
\begin{gathered}
 \left[X_1, X_2\right]=\sqrt{22} X_3, \hspace{3mm} [X_1, X_3]=6 X_4, \hspace{3mm} [X_1, X_4]=\sqrt{22} X_5, \nonumber \\
[X_1, X_5]=\sqrt{30} X_6, \hspace{3mm} [X_2, X_3]=\sqrt{30} X_5, \hspace{3mm} [X_2, X_4]= 5 X_6,  \nonumber
\end{gathered}
\end{equation}
defines a nilsoliton metric Lie algebra, previously found in \cite{Will}.
\end{example}

\subsection{Classifications} \label{class}

Classification results for dimensions 8 and 9 appear in Table \ref{tab:8DimensionalNilsolitonMetricLieAlgebras} and Table \ref{tab:9DimensionalNilsolitonMetricLieAlgebras}  respectively. We use vector notation to represent Lie algebra structures. For example, the list 
\begin{equation}
(0,0,0,0,\sqrt{45}. 23, \sqrt{14}. 24, \sqrt{91}. 25+ \sqrt{91}. 34, \sqrt{136}. 17+ \sqrt{29}. 26+ \sqrt{14}. 35, \sqrt{104}. 18+ 4. 27) \nonumber
\end{equation}
in the first row of Table \ref{tab:9DimensionalNilsolitonMetricLieAlgebras} is meant to encode the metric Lie algebra $(n,<,>)$ with orthonormal basis $B= \{ X_i \}_{i=1}^8$ and bracket relations 
\begin{equation}
\begin{gathered}
\hspace{1mm} [X_2 , X_3] =\sqrt{45} X_5 \hspace{1cm} [X_2 , X_4]=\sqrt{14} X_6 \hspace{1cm} [X_2 , X_5]=\sqrt{91} X_7 \nonumber \\
[X_3 , X_4]= \sqrt{91} X_7 \hspace{1cm} [X_1 , X_7]=\sqrt{136} X_8 \hspace{1cm} [X_2 , X_6]=\sqrt{29} X_8 \nonumber \\
[X_3 , X_5]=\sqrt{14} X_8 \hspace{1cm} [X_1 , X_8]=\sqrt{104} X_9 \hspace{1cm} [X_2 , X_7]= 4 X_9. 
\end{gathered}
\end {equation} 
\
 
\begin{table}[ht]
\tiny
	\centering
		\begin{tabular}{|l|l|l|l|}
\hline
\multicolumn{1}{|c}{}&
\multicolumn{1}{|c}{Lie Bracket}&
\multicolumn{1}{|c}{index }&
\multicolumn{1}{|c|}{Nullity}\\\hline
1  & (0,0,$\sqrt{\frac{274}{2223}}$.12, $\sqrt{\frac{99}{764}}.13$, $\sqrt{\frac{527}{8179}}.14$+ $\sqrt{\frac{1532}{9311}}.23$, $\sqrt{\frac{101}{2154}}$.15+ $\sqrt{\frac{250},{4199}}.24$, $\sqrt{\frac{150}{3151}}$.16 + $\sqrt{\frac{110}{367}}$.25+ $\sqrt{\frac{110}{367}}$.34, $\sqrt{\frac{7}{17}}$.17) & 6 & 3  \\\hline
2  & (0,0,$\sqrt{\frac{82}{7253}}$.12, $\sqrt{\frac{4}{34}}.13$, $\sqrt{\frac{97}{299}}.14$+ $\sqrt{\frac{6}{34}}.23$, $\sqrt{\frac{1}{34}}$.15+ $\sqrt{\frac{9}{34}}.24$, $\sqrt{\frac{10}{34}}$.16 + $\sqrt{\frac{1998}{12097}}$.25+ $\sqrt{\frac{397}{13917}}$.34, $\sqrt{\frac{728}{2883}}$.17 + $\sqrt{\frac{637}{4000}}$.35) & 6 & 4  \\\hline
3  & (0,0,$\sqrt{\frac{163}{702}}$.12, $\sqrt{\frac{263}{2572}}.13$, $\sqrt{\frac{343}{1334}}.14$+ $\sqrt{\frac{547}{2871}}.23$, $\sqrt{\frac{175}{2358}}$.15+ $\sqrt{\frac{160}{1149}}.$24, $\sqrt{\frac{388}{2509}}$.16 + $\sqrt{\frac{518}{6185}}$.25, $\sqrt{\frac{43}{13870}}$.17 + $\sqrt{\frac{915}{2239}}$.35) & 6 & 3  \\\hline
\end{tabular}
\caption{8-dimensional nilsoliton metric Lie algebras }
\label{tab:8DimensionalNilsolitonMetricLieAlgebras}
\end{table}
\
\begin{table}[ht]
\tiny
	\centering
		\begin{tabular}{|l|l|l|l|}
\hline
\multicolumn{1}{|c}{}&
\multicolumn{1}{|c}{Lie Bracket}&
\multicolumn{1}{|c}{index }&
\multicolumn{1}{|c|}{Nullity}\\\hline
1  & (0,0,0,0,$\sqrt{45}$.23, $\sqrt{14}.24$, $\sqrt{91}.25$+ $\sqrt{91}.34$, $\sqrt{136}$.17+ $\sqrt{29}.26$+ $\sqrt{14}.35$, $\sqrt{104}.18$+ 4.27) & 4 & 1  \\\hline
2  & (0,0,1.12,0,1.14+$\sqrt{3}$.23, 0, $\sqrt{6}.25$+ $\sqrt{6}$.34, $\sqrt{8}$.17+ 1.26+ $\sqrt{2}$.35, $\sqrt{6}.18$+ $\sqrt{2}$27) & 5 & 2  \\\hline
3  & (0,0,$\sqrt{21}$.12,0,$\sqrt{21}$.14+$\sqrt{39}$.23, 0, $\sqrt{13}$.16+ $\sqrt{70}$.25-$\sqrt{70}$.34, $\sqrt{88}$.17+ $\sqrt{42}$.35+, $\sqrt{65}$.18+ $\sqrt{39}$ .27) & 5 & 2  \\\hline
\end{tabular}
\caption{9-dimensional nilsoliton metric Lie algebras of nullity 1 and 2}
\label{tab:9DimensionalNilsolitonMetricLieAlgebras}
\end{table}

\subsubsection{Candidates of Nilsoliton Metrics }

Table \ref{tab:Numberof9DimensionalPossibleNilsolitonMetricLieAlgebras} illustrates how many possible candidates of Lie algebras appear in dimension $9$ up to the nullity of its Gram matrix. The algebras that illustrated in  Table \ref{tab:8DimensionalPossibleNilsolitonMetricLieAlgebras} are possible candidates of nilsoliton metric Lie algebras with ordered type of derivations in dimension $8$.  Here, as an example we give $Nullity = 3$,  $6$ and $8$ case in Tables \ref{tab:9dimensionalcandidatesofnullity6}, \ref{tab:9DimensionalPossibleNilsolitonMetricLieAlgebras2}, and \ref{tab:9DimensionalPossibleNilsolitonMetricLieAlgebras3} for dimension {\em nine}.
\
\begin{table}[ht]
\tiny
	\centering
		\begin{tabular}{|l|l|l|l|}
\hline
\multicolumn{1}{|c}{}&
\multicolumn{1}{|c}{Lie Bracket}&
\multicolumn{1}{|c}{index }&
\multicolumn{1}{|c|}{Nullity}\\\hline
1  & (0,0,1.12, 1.13, 1.14 + 1.23, 1.15+ 1.24, 1.16 + 1.34, 1.17 + 1.26 + 1.35) & 6 & 4  \\\hline
2  & (0,0,1.12, 1.13, 1.14 + 1.23, 1.15+ 1.24, 1.16 + 1.25, 1.17 + 1.26 + 1.35) & 6 & 4  \\\hline
3  & (0,0,1.12, 1.13, 1.14 + 1.23, 1.15+ 1.24, 1.16 + 1.25 + 1.34, 1.17 + 1.26) & 6 & 4  \\\hline
4  & (0,0,1.12, 1.13, 1.14 + 1.23, 1.15+ 1.24, 1.16 + 1.25 + 1.34, 1.17 + 1.26 + 1.35) & 6 & 5  \\\hline
\end{tabular}
\caption{8-dimensional nilsoliton metric Lie algebra candidates}
\label{tab:8DimensionalPossibleNilsolitonMetricLieAlgebras}
\end{table}

\section{Conclusion}
\label{conc}
In this work we have focused on nilpotent metric Lie algebras of dimension eight, and nine with ordered type of 
derivations. We have given specific Jacobi identity conditions for Lie algebras which  allowed us to simplify the 
Jacobi identity condition. 
We have classified nilsoliton metric Lie algebras for the corresponding Gram matrix $U$ being  invertible and
 noninvertible.  For Dimension 8, we have focused on Nilsoliton metric Lie algebras with uninvertable Gram matrix which 
leads to more than one solution for $Uv=[1]$. We have proved that if the nilpotent Lie algebra admits a soliton metric
 with corresponding Gram matrix  being noninvertible, all the solutions of $Uv=[1]$ corresponds to a unique derivation.
This theorem has allowed us to omit
several cases that come from non ordered eigenvalues without considering Jacobi condition. 
 Moreover, we have classified some nilsoliton metric Lie algebras with derivation  types $1<2<...<n$ and provided
 some candidates that may be classified. We are currently working on an algorithm that provides a full
list of classifications for dimensions {\em eight} and {\em nine}.

\section{Appendix} 


\begin{table}[ht]
\tiny
	\centering
		\begin{tabular}{|l|l|}
\hline
\multicolumn{1}{|c}{Nullity }&
\multicolumn{1}{|c|}{Number of Lie algebras}\\\hline
3  & 98 \\\hline
4  & 81 \\\hline
5  & 45 \\\hline
6  & 22 \\\hline
7  & 7  \\\hline
8  & 1  \\\hline
\end{tabular}
\caption{\small{Number of 9-dimensional  nilsoliton metric Lie algebra candidates}}
\label{tab:Numberof9DimensionalPossibleNilsolitonMetricLieAlgebras}
\end{table}


\begin{table}[ht]
\tiny
	\centering
\begin{tabular}{|l|l|l|l|}
\hline
\multicolumn{1}{|c}{}&
\multicolumn{1}{|c}{Lie Bracket}&
\multicolumn{1}{|c}{index }&
\multicolumn{1}{|c|}{Nullity}\\\hline
1 & (0, 0, 1.12, 1.13, 1.14+ 1.23, 0, 1.16+ 1.25+ 1.34, 1.17+ 1.26+ 1.35, 1.18+ 1.27+ 1.36+ 1.45) & 6 & 6  \\\hline
2 & (0, 0, 1.12, 1.13, 1.14, 1.15, 1.16+ 1.25+ 1.34, 1.17+ 1.26+ 1.35, 1.18+ 1.27+ 1.36+ 1.45) & 7 & 6  \\\hline
3 & (0, 0, 1.12, 1.13, 1.14+ 1.23, 1.15+ 1.24, 1.16+ 1.34, 1.17+ 1.35, 1.18+ 1.27+ 1.36+ 1.45) & 7 & 6  \\\hline
4 & (0, 0, 1.12, 1.13, 1.14+ 1.23, 1.15+ 1.24, 1.16+ 1.34, 1.17+ 1.26+ 1.35, 1.18+ 1.36+ 1.45) & 7 & 6  \\\hline
5 & (0, 0, 1.12, 1.13, 1.14+ 1.23, 1.15+ 1.24, 1.16+ 1.34, 1.17+ 1.26, 1.18+ 1.27+ 1.36+ 1.45) & 7 & 6  \\\hline
6 & (0, 0, 1.12, 1.13, 1.14+ 1.23, 1.15+ 1.24, 1.16+ 1.34, 1.17+ 1.26+ 1.35, 1.18+ 1.27+ 1.45) & 7 & 6  \\\hline
7 & (0, 0, 1.12, 1.13, 1.14+ 1.23, 1.15+ 1.24, 1.16+ 1.34, 1.17+ 1.26+ 1.35, 1.18+ 1.27+ 1.36) & 7 & 6  \\\hline
8 & (0, 0, 1.12, 1.13, 1.14+ 1.23, 1.15+ 1.24, 1.16+ 1.25+ 1.34, 1.17+ 1.35, 1.18+ 1.36+ 1.45) & 7 & 6  \\\hline
9 & (0, 0, 1.12, 1.13, 1.14+ 1.23, 1.15+ 1.24, 1.16+ 1.25, 1.17+ 1.35, 1.18+ 1.27+ 1.36+ 1.45) & 7 & 6  \\\hline
10 & (0, 0, 1.12, 1.13, 1.14+ 1.23, 1.15+ 1.24, 1.16+ 1.25+ 1.34, 1.17, 1.18+ 1.27+ 1.36+ 1.45) & 7 & 6  \\\hline
11 & (0, 0, 1.12, 1.13, 1.14+ 1.23, 1.15+ 1.24, 1.16+ 1.25+ 1.34, 1.17+ 1.35, 1.18+ 1.27+ 1.45) & 7 & 6  \\\hline
12 & (0, 0, 1.12, 1.13, 1.14+ 1.23, 1.15+ 1.24, 1.16+ 1.25+ 1.34, 1.17+ 1.35, 1.18+ 1.27+ 1.36) & 7 & 6  \\\hline
13 & (0, 0, 1.12, 1.13, 1.14+ 1.23, 1.15+ 1.24, 1.16+ 1.25, 1.17+ 1.26+ 1.35, 1.18+ 1.36+ 1.45) & 7 & 6  \\\hline
14 & (0, 0, 1.12, 1.13, 1.14+ 1.23, 1.15+ 1.24, 1.16+ 1.25+ 1.34, 1.17+ 1.26, 1.18+ 1.36+ 1.45) & 7 & 6  \\\hline
15 & (0, 0, 1.12, 1.13, 1.14+ 1.23, 1.15+ 1.24, 1.16+ 1.25+ 1.34, 1.17+ 1.26+ 1.35, 1.18+ 1.45) & 7 & 6  \\\hline
16 & (0, 0, 1.12, 1.13, 1.14+ 1.23, 1.15+ 1.24, 1.16+ 1.25+ 1.34, 1.17+ 1.26+ 1.35, 1.18+ 1.36) & 7 & 6  \\\hline
17 & (0, 0, 1.12, 1.13, 1.14+ 1.23, 1.15+ 1.24, 1.16+ 1.25, 1.17+ 1.26, 1.18+ 1.27+ 1.36+ 1.45) & 7 & 6  \\\hline
18 & (0, 0, 1.12, 1.13, 1.14+ 1.23, 1.15+ 1.24, 1.16+ 1.25, 1.17+ 1.26+ 1.35, 1.18+ 1.27+ 1.45) & 7 & 6  \\\hline
19 & (0, 0, 1.12, 1.13, 1.14+ 1.23, 1.15+ 1.24, 1.16+ 1.25, 1.17+ 1.26+ 1.35, 1.18+ 1.27+ 1.36) & 7 & 6  \\\hline
20 & (0, 0, 1.12, 1.13, 1.14+ 1.23, 1.15+ 1.24, 1.16+ 1.25+ 1.34, 1.17+ 1.26, 1.18+ 1.27+ 1.45) & 7 & 6  \\\hline
21 & (0, 0, 1.12, 1.13, 1.14+ 1.23, 1.15+ 1.24, 1.16+ 1.25+ 1.34, 1.17+ 1.26, 1.18+ 1.27+ 1.36) & 7 & 6  \\\hline
22 & (0, 0, 1.12, 1.13, 1.14+ 1.23, 1.15+ 1.24, 1.16+ 1.25+ 1.34, 1.17+ 1.26+ 1.35, 1.18+ 1.27) & 7 & 6  \\\hline
23 & (0, 0, 1.12, 1.13, 1.14+ 1.23, 1.15+ 1.24, 1.16+ 1.25+ 1.34, 1.17+ 1.26+ 1.35, 1.18+ 1.27+ 1.36+ 1.45) & 7 & 8  \\\hline
\end{tabular}
\caption{\small 9-dimensional candidates of nullity 6 and 8}
\label{tab:9dimensionalcandidatesofnullity6}
\end{table}

\begin{table}
\tiny
	\centering
\begin{tabular}{|l|l|l|l|}
\hline
\multicolumn{1}{|c}{}&
\multicolumn{1}{|c}{Lie Bracket}&
\multicolumn{1}{|c}{index }&
\multicolumn{1}{|c|}{Nullity}\\\hline
1  & (0,0,0,0, 1.23, 1.24, 1.25 +1.34, 1.17+ 1.26+ 1.35, 1.18, 1.27 +1.36 +1.45) & 4 & 3  \\\hline
2  & (0,0,1.12, 0, 1.23, 1.24, 1.16 + 1.25+ 1.34, 1.17, 1.18+ 1.27+ 1.36+ 1.45) & 5 & 3  \\\hline
3  & (0,0,1.12, 0, 1.23, 1.24, 1.16 + 1.25+ 1.34, 1.17 +1.35, 1.18+ 1.27+ 1.45) & 5 & 3  \\\hline
4  & (0,0,1.12, 0, 1.23, 1.24, 1.16 + 1.25+ 1.34, 1.17 +1.35, 1.18+ 1.27+ 1.36) & 5 & 3  \\\hline
5  & (0,0,1.12, 0, 1.23, 1.24, 1.16 + 1.25+ 1.34, 1.17 +1.26, 1.18+ 1.36+ 1.45) & 5 & 3  \\\hline
6  & (0,0,1.12, 0, 1.23, 1.24, 1.16 + 1.25+ 1.34, 1.17 +1.26+ 1.35, 1.18+ 1.45) & 5 & 3  \\\hline
7  & (0,0,1.12, 0, 1.23, 1.24, 1.16 + 1.25+ 1.34, 1.17 +1.26+ 1.35, 1.18+ 1.36) & 5 & 3  \\\hline
8  & (0,0,1.12, 0, 1.23, 1.24, 1.16 + 1.25+ 1.34, 1.17 +1.26, 1.18+ 1.27+ 1.45) & 5 & 3  \\\hline
9  & (0,0,1.12, 0, 1.23, 1.24, 1.16 + 1.25+ 1.34, 1.17 +1.26, 1.18+ 1.27+ 1.36) & 5 & 3  \\\hline
10 & (0,0,1.12, 0, 1.23, 1.24, 1.16 + 1.25+ 1.34, 1.17 +1.26+ 1.35, 1.18+ 1.27) & 5 & 3  \\\hline
11 & (0,0,1.12, 0, 1.14+ 1.23, 0, 1.25+ 1.34, 1.17 +1.26, 1.18+ 1.27+ 1.36+ 1.45) & 5 & 3  \\\hline
12 & (0,0,1.12, 0, 1.14+ 1.23, 0, 1.25+ 1.34, 1.17 +1.26+ 1.35, 1.18+ 1.27+ 1.45) & 5 & 3  \\\hline
13 & (0,0,1.12, 0, 1.14+ 1.23, 0, 1.25+ 1.34, 1.17 +1.26+ 1.35, 1.18+ 1.27+ 1.36) & 5 & 3  \\\hline
14 & (0,0,1.12, 0, 1.14+ 1.23, 1.24, 1.25+ 1.34, 1.17, 1.18+ 1.27+ 1.36+ 1.45) & 5 & 3  \\\hline
15 & (0,0,1.12, 0, 1.14+ 1.23, 1.24, 1.25+ 1.34, 1.17+ 1.35, 1.18+ 1.27+ 1.36) & 5 & 3  \\\hline
16 & (0,0,1.12, 0, 1.14+ 1.23, 1.24, 1.25+ 1.34, 1.17+ 1.26, 1.18+ 1.36+ 1.45) & 5 & 3  \\\hline
17 & (0,0,1.12, 0, 1.14+ 1.23, 1.24, 1.25+ 1.34, 1.17 +1.26+ 1.35, 1.18+ 1.45) & 5 & 3  \\\hline
18 & (0,0,1.12, 0, 1.14+ 1.23, 1.24, 1.25+ 1.34, 1.17 +1.26+ 1.35, 1.18+ 1.45) & 5 & 3  \\\hline
19 & (0,0,1.12, 0, 1.14+ 1.23, 1.24, 1.25+ 1.34, 1.17 +1.26, 1.18+ 1.27+ 1.45) & 5 & 3  \\\hline
20 & (0,0,1.12, 0, 1.14+ 1.23, 1.24, 1.25+ 1.34, 1.17 +1.26, 1.18+ 1.27+ 1.36) & 5 & 3  \\\hline
21 & (0,0,1.12, 0, 1.14+ 1.23, 1.24, 1.25+ 1.34, 1.17 +1.26+ 1.35, 1.18+ 1.27) & 5 & 3  \\\hline
22 & (0,0,1.12, 0, 1.14+ 1.23, 0, 1.16+ 1.25+ 1.34, 1.17, 1.18+ 1.27+ 1.36+ 1.45) & 5 & 3  \\\hline
23 & (0,0,1.12, 0, 1.14+ 1.23, 0, 1.16+ 1.25+ 1.34, 1.17+ 1.35, 1.18+ 1.27+ 1.45) & 5 & 3  \\\hline
24 & (0,0,1.12, 0, 1.14+ 1.23, 0, 1.16+ 1.25+ 1.34, 1.17+ 1.35, 1.18+ 1.27+ 1.36) & 5 & 3  \\\hline
25 & (0,0,1.12, 0, 1.14+ 1.23, 0, 1.16+ 1.25+ 1.34, 1.17+ 1.26, 1.18+ 1.36+ 1.45) & 5 & 3  \\\hline
26 & (0,0,1.12, 0, 1.14+ 1.23, 0, 1.16+ 1.25+ 1.34, 1.17+ 1.26+ 1.35, 1.18+ 1.45) & 5 & 3  \\\hline
27 & (0,0,1.12, 0, 1.14+ 1.23, 0, 1.16+ 1.25+ 1.34, 1.17+ 1.26, 1.18+ 1.27+ 1.45) & 5 & 3  \\\hline
28 & (0,0,1.12, 0, 1.14+ 1.23, 0, 1.16+ 1.25+ 1.34, 1.17+ 1.26, 1.18+ 1.27+ 1.36) & 5 & 3  \\\hline
29 & (0,0,1.12, 0, 1.14+ 1.23, 0, 1.16+ 1.25+ 1.34, 1.17+ 1.26+ 1.35, 1.18+ 1.27) & 5 & 3  \\\hline
30 & (0,0,1.12, 0, 1.14+ 1.23, 1.24, 1.16+ 1.34, 1.17+ 1.35, 1.18+ 1.27+ 1.45) & 4 & 3  \\\hline
31 & (0,0,1.12, 0, 1.14+ 1.23, 1.24, 1.16+ 1.34, 1.17+ 1.35, 1.18+ 1.27+ 1.36) & 4 & 3  \\\hline
32 & (0,0,1.12, 0, 1.14+ 1.23, 1.24, 1.16+ 1.34, 1.17+ 1.26+ 1.35, 1.18+ 1.45) & 4 & 3  \\\hline
33 & (0,0,1.12, 0, 1.14+ 1.23, 1.24, 1.16+ 1.34, 1.17+ 1.26+ 1.35, 1.18+ 1.36) & 4 & 3  \\\hline
34 & (0,0,1.12, 0, 1.14+ 1.23, 1.24, 1.16+ 1.34, 1.17+ 1.26+ 1.35, 1.18+ 1.27) & 4 & 3  \\\hline
35 & (0,0,1.12, 1.13, 1.23, 0, 1.16+ 1.25, 1.17+ 1.26+ 1.35, 1.27+ 1.36+ 1.45) & 4 & 3  \\\hline
36 & (0,0,1.12, 1.13, 1.14+ 1.23, 0, 1.25+ 1.34, 1.26+ 1.35, 1.18+ 1.27+ 1.45) & 5 & 3  \\\hline
37 & (0,0,1.12, 1.13, 1.14+ 1.23, 0, 1.25+ 1.34, 1.26+ 1.35, 1.18+ 1.27+ 1.36) & 5 & 3  \\\hline
38 & (0,0,1.12, 1.13, 1.14+ 1.23, 0, 1.25+ 1.34, 1.17+ 1.26, 1.27+ 1.36+ 1.45) & 5 & 3  \\\hline
39 & (0,0,1.12, 1.13, 1.14+ 1.23, 0, 1.25+ 1.34, 1.17+ 1.26+ 1.35, 1.27+ 1.45) & 5 & 3  \\\hline
40 & (0,0,1.12, 1.13, 1.14+ 1.23, 0, 1.25+ 1.34, 1.17+ 1.26+ 1.35, 1.27+ 1.36) & 5 & 3  \\\hline
41 & (0,0,1.12, 1.13, 1.14, 0, 1.25+ 1.34, 1.17+ 1.26, 1.18+ 1.27+ 1.36+ 1.45) & 5 & 3  \\\hline
42 & (0,0,1.12, 1.13, 1.14, 0, 1.25+ 1.34, 1.17+ 1.26+ 1.35, 1.18+ 1.27+ 1.45) & 6 & 3  \\\hline
43 & (0,0,1.12, 1.13, 1.14, 0, 1.25+ 1.34, 1.17+ 1.26+ 1.35, 1.18+ 1.27+ 1.36) & 6 & 3  \\\hline
44 & (0,0,1.12, 1.13, 1.14+ 1.23, 0, 1.25+ 1.34, 1.17+ 1.26, 1.18+ 1.36+ 1.45) & 6 & 3  \\\hline
45 & (0,0,1.12, 1.13, 1.14+ 1.23, 0, 1.25+ 1.34, 1.17+ 1.26+ 1.35, 1.18+ 1.45) & 6 & 3  \\\hline
46 & (0,0,1.12, 1.13, 1.14+ 1.23, 0, 1.25+ 1.34, 1.17+ 1.26, 1.18+ 1.27+ 1.45) & 6 & 3  \\\hline
47 & (0,0,1.12, 1.13, 1.14+ 1.23, 0, 1.25+ 1.34, 1.17+ 1.26, 1.18+ 1.27+ 1.36) & 6 & 3  \\\hline
48 & (0,0,1.12, 1.13, 1.14+ 1.23, 0, 1.25+ 1.34, 1.17+ 1.26+ 1.35, 1.18+ 1.27) & 6 & 3  \\\hline
49 & (0,0,1.12, 1.13, 1.14+ 1.23, 0, 1.16+ 1.25+ 1.34, 1.26+ 1.35, 1.27+ 1.45) & 5 & 3  \\\hline
50 & (0,0,1.12, 1.13, 1.14+ 1.23, 0, 1.16+ 1.25+ 1.34, 1.26+ 1.35, 1.27+ 1.36) & 5 & 3  \\\hline
51 & (0,0,1.12, 1.13, 1.14+ 1.23, 0, 1.16+ 1.25+ 1.34, 1.35, 1.18+ 1.27+ 1.45) & 5 & 3  \\\hline
52 & (0,0,1.12, 1.13, 1.14+ 1.23, 0, 1.16+ 1.25+ 1.34, 1.35, 1.18+ 1.27+ 1.36) & 5 & 3  \\\hline
53 & (0,0,1.12, 1.13, 1.14+ 1.23, 0, 1.16+ 1.25+ 1.34, 1.26+ 1.35, 1.18+ 1.45) & 5 & 3  \\\hline
54 & (0,0,1.12, 1.13, 1.14+ 1.23, 0, 1.16+ 1.25+ 1.34, 1.26+ 1.35, 1.18+ 1.36) & 5 & 3  \\\hline
55 & (0,0,1.12, 1.13, 1.14+ 1.23, 0, 1.16+ 1.25+ 1.34, 1.26+ 1.35, 1.18+ 1.27) & 5 & 3  \\\hline
56 & (0,0,1.12, 1.13, 1.14, 0, 1.16+ 1.25+ 1.34, 1.17+ 1.26, 1.27+ 1.36+ 1.45) & 5 & 3  \\\hline
57 & (0,0,1.12, 1.13, 1.14, 0, 1.16+ 1.25+ 1.34, 1.17+ 1.26+ 1.35, 1.27+ 1.45) & 5 & 3  \\\hline
58 & (0,0,1.12, 1.13, 1.14+ 1.23, 0, 1.16+ 1.25+ 1.34, 1.17, 1.27+ 1.36+ 1.45) & 5 & 3  \\\hline
59 & (0,0,1.12, 1.13, 1.14+ 1.23, 0, 1.16+ 1.25+ 1.34, 1.17+ 1.35, 1.27+ 1.45) & 5 & 3  \\\hline
60 & (0,0,1.12, 1.13, 1.14, 0, 1.16+ 1.25+ 1.34, 1.17, 1.18+ 1.27+ 1.36+ 1.45) & 6 & 3  \\\hline
61 & (0,0,1.12, 1.13, 1.14, 0, 1.16+ 1.25+ 1.34, 1.17+ 1.35, 1.18+ 1.27+ 1.45) & 6 & 3  \\\hline
62 & (0,0,1.12, 1.13, 1.14, 0, 1.16+ 1.25+ 1.34, 1.17+ 1.35, 1.18+ 1.27+ 1.36) & 6 & 3  \\\hline
63 & (0,0,1.12, 1.13, 1.14, 0, 1.16+ 1.25+ 1.34, 1.17+ 1.26, 1.18+ 1.27+ 1.45) & 6 & 3  \\\hline
64 & (0,0,1.12, 1.13, 1.14, 0, 1.16+ 1.25+ 1.34, 1.17+ 1.26, 1.18+ 1.27+ 1.36) & 6 & 3  \\\hline
65 & (0,0,1.12, 1.13, 1.14, 0, 1.16+ 1.25+ 1.34, 1.17+ 1.26+ 1.35, 1.18+ 1.27) & 6 & 3  \\\hline
66 & (0,0,1.12, 1.13, 1.14+ 1.23, 0, 1.16+ 1.25+ 1.34, 1.17, 1.18+ 1.36+ 1.45) & 6 & 3  \\\hline
67 & (0,0,1.12, 1.13, 1.14+ 1.23, 0, 1.16+ 1.25+ 1.34, 1.17+ 1.35, 1.18+ 1.45) & 6 & 3  \\\hline
68 & (0,0,1.12, 1.13, 1.14+ 1.23, 0, 1.16+ 1.25+ 1.34, 1.17+ 1.35, 1.18+ 1.45) & 6 & 3  \\\hline
\end{tabular}
\caption{\small 9-dimensional nilsoliton metric Lie algebra candidates of nullity 3}
\label{tab:9DimensionalPossibleNilsolitonMetricLieAlgebras2}
\end{table}

\begin{table}
\tiny
	\centering
\begin{tabular}{|l|l|l|l|}
\hline
\multicolumn{1}{|c}{}&
\multicolumn{1}{|c}{Lie Bracket}&
\multicolumn{1}{|c}{index }&
\multicolumn{1}{|c|}{Nullity}\\\hline
69 & (0,0,1.12, 1.13, 1.14+ 1.23, 0, 1.16+ 1.25+ 1.34, 1.17, 1.18+ 1.27+ 1.45) & 6 & 3  \\\hline
70 & (0,0,1.12, 1.13, 1.14+ 1.23, 0, 1.16+ 1.25+ 1.34, 1.17, 1.18+ 1.27+ 1.36) & 6 & 3  \\\hline
71 & (0,0,1.12, 1.13, 1.14+ 1.23, 0, 1.16+ 1.25+ 1.34, 1.17+ 1.35, 1.18+ 1.27) & 6 & 3  \\\hline
72 & (0,0,1.12, 1.13, 1.14+ 1.23, 0, 1.16+ 1.25+ 1.34, 1.17+ 1.26, 1.18+ 1.45) & 6 & 3  \\\hline
73 & (0,0,1.12, 1.13, 1.14+ 1.23, 0, 1.16+ 1.25+ 1.34, 1.17+ 1.26+ 1.35, 1.18) & 6 & 3  \\\hline
74 & (0,0,1.12, 1.13, 1.14, 1.15, 1.25+ 1.34, 1.26, 1.18+ 1.27+ 1.36+ 1.45) & 6 & 3  \\\hline
75 & (0,0,1.12, 1.13, 1.14, 1.15, 1.25+ 1.34, 1.26+ 1.35, 1.18+ 1.27+ 1.45) & 6 & 3  \\\hline
76 & (0,0,1.12, 1.13, 1.14, 1.15, 1.25+ 1.34, 1.26+ 1.35, 1.18+ 1.27+ 1.36) & 6 & 3  \\\hline
77 & (0,0,1.12, 1.13, 1.14, 1.15, 1.25+ 1.34, 1.17, 1.18+ 1.27+ 1.36+ 1.45) & 6 & 3  \\\hline
78 & (0,0,1.12, 1.13, 1.14, 1.15, 1.25+ 1.34, 1.17+ 1.35, 1.18+ 1.27+ 1.36) & 6 & 3  \\\hline
79 & (0,0,1.12, 1.13, 1.14, 1.15, 1.16+ 1.25+ 1.34, 1.26, 1.27+ 1.36+ 1.45) & 7 & 3  \\\hline
80 & (0,0,1.12, 1.13, 1.14, 1.15, 1.16+ 1.25+ 1.34, 1.26+ 1.35, 1.27+ 1.45) & 7 & 3  \\\hline
81 & (0,0,1.12, 1.13, 1.14, 1.15, 1.16+ 1.25+ 1.34, 1.26+ 1.35, 1.27+ 1.36) & 7 & 3  \\\hline
82 & (0,0,1.12, 1.13, 1.14, 1.15, 1.16+ 1.25+ 1.34, 1.35, 1.18+ 1.27+ 1.45) & 7 & 3  \\\hline
83 & (0,0,1.12, 1.13, 1.14, 1.15, 1.16+ 1.25+ 1.34, 1.35, 1.18+ 1.27+ 1.36) & 7 & 3  \\\hline
84 & (0,0,1.12, 1.13, 1.14, 1.15, 1.16+ 1.25+ 1.34, 1.26, 1.18+ 1.27+ 1.45) & 7 & 3  \\\hline
85 & (0,0,1.12, 1.13, 1.14, 1.15, 1.16+ 1.25+ 1.34, 1.26, 1.18+ 1.27+ 1.36) & 7 & 3  \\\hline
86 & (0,0,1.12, 1.13, 1.14, 1.15, 1.16+ 1.25+ 1.34, 1.26+ 1.35, 1.18+ 1.27) & 7 & 3  \\\hline
87 & (0,0,1.12, 1.13, 1.14, 1.15, 1.16+ 1.25+ 1.34, 1.17, 1.27+ 1.36+ 1.45) & 6 & 3  \\\hline
88 & (0,0,1.12, 1.13, 1.14, 1.15, 1.16+ 1.25+ 1.34, 1.17+ 1.35, 1.27+ 1.45) & 6 & 3  \\\hline
89 & (0,0,1.12, 1.13, 1.14, 1.15, 1.16+ 1.25+ 1.34, 1.17+ 1.35, 1.27+ 1.36) & 6 & 3  \\\hline
90 & (0,0,1.12, 1.13, 1.14, 1.15, 1.16+ 1.25+ 1.34, 1.17, 1.18+ 1.27+ 1.45) & 7 & 3  \\\hline
91 & (0,0,1.12, 1.13, 1.14, 1.15, 1.16+ 1.25+ 1.34, 1.17, 1.18+ 1.27+ 1.36) & 7 & 3  \\\hline
92 & (0,0,1.12, 1.13, 1.14, 1.15, 1.16+ 1.25+ 1.34, 1.17+ 1.35, 1.18+ 1.27) & 7 & 3  \\\hline
93 & (0,0,1.12, 1.13, 1.14+ 1.23, 1.15+ 1.24, 1.16+ 1.34, 1.17+ 1.35, 1.18) & 7 & 3  \\\hline
94 & (0,0,1.12, 1.13, 1.14+ 1.23, 1.15+ 1.24, 1.16+ 1.34, 1.17+ 1.26, 1.18) & 7 & 3  \\\hline
95 & (0,0,1.12, 1.13, 1.14+ 1.23, 1.15+ 1.24, 1.16+ 1.25, 1.17, 1.18+ 1.45) & 7 & 3  \\\hline
96 & (0,0,1.12, 1.13, 1.14+ 1.23, 1.15+ 1.24, 1.16+ 1.25, 1.17, 1.18+ 1.36) & 7 & 3  \\\hline
97 & (0,0,1.12, 1.13, 1.14+ 1.23, 1.15+ 1.24, 1.16+ 1.25, 1.17+ 1.35, 1.18) & 7 & 3  \\\hline
98 & (0,0,1.12, 1.13, 1.14+ 1.23, 1.15+ 1.24, 1.16+ 1.25+ 1.34, 1.17, 1.18) & 7 & 3  \\\hline
\end{tabular}
\caption{\small 9-dimensional nilsoliton metric Lie algebra candidates of nullity 3}
\label{tab:9DimensionalPossibleNilsolitonMetricLieAlgebras3}
\end{table}

\vspace{3cm} 
\section{Acknowledgement} Author declares that she has no competing interest.

\end{document}